\documentclass[12pt]{amsart}
\usepackage{amsmath}
\usepackage{amsthm}
\usepackage{amssymb}
\usepackage{tikz}
\usepackage{hyperref}
\usepackage[left=2.5cm,right=2.5cm,top=3cm,bottom=3cm]{geometry}
\usepackage{eucal}
\usepackage{palatino}
\usepackage{euler}
\usepackage{mathrsfs}
\usepackage{amssymb}
\usepackage{amscd}
\usepackage{latexsym}
\usepackage{epsfig}
\usepackage{graphicx}
\usepackage{amsfonts}
\usepackage{psfrag}
\usepackage{caption}
\usepackage{rotating}
\usepackage{mathtools}
\usepackage{fullpage}
\usepackage{etoolbox} 

\usepackage[normalem]{ulem}

\usepackage{pifont}
\usepackage{epsfig}

\usepackage{url}
\usepackage{cite}
\usepackage{dsfont}
\usepackage{array}

\newtheorem{theorem}{Theorem}

\newtheorem{proposition}[theorem]{Proposition}
\newtheorem{corollary}[theorem]{Corollary}
\newtheorem{lemma}[theorem]{Lemma}

\theoremstyle{remark}
\newtheorem*{remark}{Remark}

\DeclareMathOperator{\GL}{GL}
\DeclareMathOperator{\SL}{SL}

\DeclareMathOperator{\Adj}{Ad}

\DeclareMathOperator{\adj}{ad}

\DeclareMathOperator{\Hom}{Hom}

\DeclareMathOperator{\Sp}{Sp}
\DeclareMathOperator{\SO}{SO}
\newtoggle{showColor}
\newtoggle{showNotes}
\toggletrue{showColor}
\togglefalse{showNotes}

\iftoggle{showNotes} {%
	\newcommand{\note}[1]{{\textcolor{red}{$\langle$#1$\rangle$}}} 
}{%
	\newcommand{\note}[1]{}
}

\title{Some results on equivariant contact geometry for partial flag varieties}
\author{Peter Crooks}
\author{Steven Rayan}
\address{Department of Mathematics, University of Toronto, Canada}
\email{~~~peter.crooks@utoronto.ca, stever@math.toronto.edu}

\begin{document}
\begin{abstract}  We study equivariant contact structures on complex projective varieties arising as partial flag varieties $G/P$, where $G$ is a connected, simply-connected complex simple group of type $ADE$ and $P$ is a parabolic subgroup.  We prove a special case of the LeBrun-Salamon conjecture for partial flag varieties of these types.  The result can be deduced from Boothby's classification of compact simply-connected complex contact manifolds with transitive action by contact automorphisms, but our proof is completely independent and relies on properties of $G$-equivariant vector bundles on $G/P$.  A byproduct of our argument is a canonical, global description of the unique $SO_{2n}(\mathbb C)$-invariant contact structure on the isotropic Grassmannian of $2$-planes in $\mathbb C^{2n}$.
\end{abstract}



\maketitle

\section{Introduction}  Fano varieties with complex contact structures have been studied enthusiastically over the last half century, in large part due to their distinguished position at the intersection of complex algebraic geometry and real differential geometry.  A compact quaternionic K\"ahler manifold with positive curvature always supports an $S^2$-bundle --- its twistor space --- the total space of which is a Fano contact manifold.  As presented in \cite{KPSW2000}, the LeBrun-Salamon conjecture \cite{LeBrunSalamon} posits that every Fano contact manifold with $b_2=1$ is a homogeneous variety, one that is isomorphic to the unique closed orbit $\mathbb{P}(\mathcal{O}_{\text{min}})$ in the projectivized (co)adjoint representation of some simple Lie group $G$.     If the conjecture were true, then every compact quaternionic K\"ahler variety with positive curvature would necessarily be homogeneous, and so progress on the LeBrun-Salamon conjecture is crucial to resolving an outstanding geometric classification problem within Riemannian geometry.  The work of Beauville \cite{Beauville1998} is the strongest evidence thus far for the validity of the conjecture.  For multiple points of view on Fano contact varieties, including the minimal rational curves and Mori theory approaches, we refer the reader to \cite{Hwang2001,Kebekus2001,Peternell2001-1,Peternell2001-2,Buczynski2010}.

On another front, complex contact manifolds have been studied in the context of equivariant geometry. Most notably, Boothby \cite{Boothby1961,Boothby1962} gives a complete classification of those compact simply-connected complex contact manifolds which are acted upon transitively by their respective groups of contact automorphisms (the so-called homogeneous complex contact manifolds). He identifies each with $\mathbb{P}(\mathcal{O}_{\text{min}})$ for a suitable simple group $G$.

\subsection{Results}  
This article presents some results on equivariant contact geometry for partial flag varieties. Firstly, we give a self-contained proof of the following special case of the LeBrun-Salamon conjecture.

\begin{theorem}\label{Main Theorem}\
Let $G$ be a connected, simply-connected complex simple group of type $ADE$, and let $X$ be a partial flag variety of $G$ with $b_2(X)=1$. If $X$ is endowed with a $G$-invariant complex contact structure, then there exists a $G$-equivariant isomorphism $X\cong\mathbb{P}(\mathcal{O}_{\text{min}})$ of contact varieties, where $\mathcal{O}_{\text{min}}$ is the minimal nilpotent orbit of $G$. 
\end{theorem}     

While Theorem \ref{Main Theorem} is deducible from Boothby's work, our argument differs significantly from that offered in \cite{Boothby1961,Boothby1962}. Our approach is instead based on a sequence of results concerning the geometry (both equivariant and non-equivariant) of partial flag varieties $G/P$, where $P$ is a parabolic subgroup.  Specifically, we prove that a $G$-invariant corank-$1$ subbundle $\mathbf{E}$ of $\mathbf{T}_{G/P}$ is completely determined as such by the isomorphism class of the quotient line bundle $\mathbf{T}_{G/P}/\mathbf{E}$ (see Proposition \ref{Uniqueness of Corank-One Subbundles}). This leads us to prove Proposition \ref{Description of the Contact Line Bundle}, which describes the contact line bundle of a $G$-invariant contact structure on $G/P$ in terms of the isomorphism between $\text{Pic}(G/P)$ and the group of $1$-dimensional $P$-representations. Proposition \ref{Description of Alpha and Lambda} and Theorem \ref{Main Theorem 2} then combine to give us the desired $G$-equivariant contact variety isomorphism between $G/P$ and $\mathbb{P}(\mathcal{O}_{\text{min}})$.        

Secondly, we offer a detailed description of the contact manifold in Theorem \ref{Main Theorem} when $G$ is of type $D_n$. This manifold is precisely the Grassmannian $Gr_B(2,\mathbb{C}^{2n})$ of those $2$-planes in $\mathbb{C}^{2n}$ which are isotropic with respect to the complex-bilinear dot product. While there are descriptions of the $\SO_{2n}(\mathbb{C})$-invariant contact distribution $\mathbf{E}$ on $Gr_B(2,\mathbb{C}^{2n})$ appearing in the literature (e.g. \cite{Hwang2001}), ours is global and canonical.  Indeed, we use the classical identification of the tangent bundle of the full Grassmannian $Gr(2,\mathbb{C}^{2n})$ with $\Hom(\mathbf{F},\mathcal O^{\oplus 2n}/\mathbf{F})$, where $\mathbf{F}$ is the tautological bundle on $Gr(2,\mathbb{C}^{2n})$. We then present $\mathbf{E}$ explicitly as a subbundle of the pullback to $Gr_B(2,\mathbb{C}^{2n})$ of  $\Hom(\mathbf{F},\mathcal O^{\oplus 2n}/\mathbf{F})$.\\             

\emph{\textbf{Acknowledgements.}}  We gratefully acknowledge the support provided by Lisa Jeffrey and John Scherk.  We also thank Steven Lu for useful discussions.  The first author was supported by NSERC CGS and OGS awards.  During this work, the second author was supported by a University of Toronto at Scarborough Postdoctoral Fellowship.

\section{Review of Properties of Fano Contact Varieties}\label{Section-Review-Fano-Contact}

Here, we review the salient features of complex contact varieties in general and Fano contact varieties in particular. Let $X$ be a smooth complex variety of complex dimension $2n+1$ for some $n\geq0$, and let $\iota:\mathbf E\hookrightarrow\mathbf T_X$ be a rank-$2n$ holomorphic subbundle of the tangent bundle $\mathbf T_X$.  We say that the pair $(X,\mathbf E)$ is \emph{contact} if, in the short exact sequence\begin{eqnarray}0\longrightarrow\mathbf E\stackrel{\iota}{\longrightarrow}\mathbf T_X\stackrel{\theta}{\longrightarrow}\mathbf L\longrightarrow0\label{ContactSES}\end{eqnarray}\noindent induced by $\iota$, the composition of the Lie bracket on sections of $\mathbf T_X$ with the quotient map $\theta$ is an $\mathbf L$-twisted bilinear form that is nondegenerate along $\mathbf E$.  In keeping with the literature, we call the subbundle $\mathbf E$ the \emph{contact distribution}; the quotient $\mathbf L$, the \emph{contact line bundle} of $(X,\mathbf E)$.  If there exists an $\mathbf E\to X$ for which the pair $(X,\mathbf E)$ is contact, we say that $X$ \emph{admits a contact structure}. 

From now on, we also assume that $X$ is projective, that $b_2(X)=1$, and that $X$ admits a contact structure, the distribution of which is  $\mathbf E$.  Let $\mathbf L$ be the associated contact line bundle.  We use $\mathbf K_X$ and $\mathbf K_X^\vee$, respectively, for the canonical and anticanonical line bundles of $X$.  In this case, $X$ is Fano with $\mbox{Pic}(X)\cong\mathbb Z$ and $\mathbf K_X^\vee\cong\mathbf L^{\otimes n+1}$. This characterization is a consequence of a theorem of Demailly (Cor. 3 in \cite{Demailly2002}), applied to an earlier result of Kebekus, Peternell, Sommese, Wi\'sniewski (Thm. 1.1 in \cite{KPSW2000}).

There are of course exactly two possibilities: either the contact line bundle $\mathbf L$ is a generator of $\mbox{Pic}(X)$ or it is not.  If it is not, then $\mathbf L$ is a holomorphic $(n+1)$-th root of $\mathbf K_X^\vee$ and $\mathbf L$ itself has nontrivial roots (namely, a generator of $\mbox{Pic}(X)$).  In this case, $X$ must be $\mathbb P^N$ for some $N$, by the well-known Kobayashi-Ochiai characterization of complex projective space \cite{KobayashiOchiai1973}.  Hence, whenever $X$ is a projective Fano contact variety with $b_2=1$ that is \emph{not} a projective space, then it must be that $\mbox{Pic}(X)=\mathbb Z\cdot[\mathbf L]$.

Taking these observations together, we have that:\begin{itemize}\item $\mathbf L$ is ample (in particular, it is the ample generator of $\mbox{Pic}(X)$ whenever $X\ncong\mathbb P^N$); and\item if $\mathbf L'$ is any other contact line bundle on $X$, then there must exist a holomorphic vector bundle isomorphism $\mathbf L\cong\mathbf L'$.\end{itemize}

The second fact is true because $\mbox{Pic}(X)\cong\mathbb Z$ and $\mathbf L$ and $\mathbf L'$ are holomorphic roots of the same line bundle (and hence $\deg\mathbf L=\deg\mathbf L'$).

\section{Partial Flag Varieties and Contact Structures}
\subsection{Basic Setup}\label{Basic Setup}

Now, we specialize to the case where $X$ arises as a partial flag variety for a simple Lie group $G$.

To be precise, let $G$ be a connected, simply-connected complex simple group with Lie algebra $\mathfrak{g}$. Fix a maximal torus $T\subseteq G$ with Lie algebra $\mathfrak{t}\subseteq\mathfrak{g}$. Let $X^*(T)$ denote the weight lattice, and let $\Delta\subseteq X^*(T)$ be the resulting collection of roots. Also, let $(\cdot,\cdot)$ be the standard inner product on $\text{span}_{\mathbb{R}}(\Delta)\subseteq\mathfrak{t}^*$. Choose collections $\Pi\subseteq\Delta^+$ of simple roots and positive roots, respectively. Since $\mathfrak{g}$ is simple, there exists a unique highest root $\lambda$ with respect to the partial order induced by the choice of $\Pi$. Let $B\subseteq G$ be the opposite Borel with respect to our choice of positive roots. After fixing a subset $S$ of $\Pi$, let $\Delta_S^+$ be the set of those positive roots expressible as $\mathbb{Z}$-linear combinations of the roots in $S$. One then has the standard parabolic subgroup $P_S$ generated by $B$ and the root subgroups coming from the roots in $\Delta_S^+$. Hence, the Lie algebra of $P_S$ is precisely $$\mathfrak{p}_S=\mathfrak{b}\oplus\bigoplus_{\beta\in\Delta_S^+}\mathfrak{g}_{\beta},$$ where $\mathfrak{b}$ is the Lie algebra of $B$. Finally, let $W=N_G(T)/T$ be the Weyl group, and let $W_S$ denote the subgroup of $W$ generated by the simple reflections $\{s_{\beta}:\beta\in\Pi\setminus S\}$.  The quotient $G/P_S$ is not only a complex projective variety, but is also Fano --- see, for instance, Thm. V.1.4 in \cite{Kollar1996}.

We will make extensive use of vector bundles on $G/P_S$ arising from $P_S$-representations via the associated bundle construction. More explicitly, suppose that $\varphi:P_S\rightarrow\GL(V)$ is a finite-dimensional complex $P_S$-representation. One has a free $P_S$-action on $G\times V$ defined by $$p\cdot (g,v)=(gp^{-1},\varphi(p)v),$$ $g\in G$, $v\in V$, $p\in P_S$, with quotient variety denoted $G\times_{P_S}V$. Note that the action of $G$ on $G\times V$ given by left-multiplication in the first factor commutes with the $P_S$-action, so that $G\times_{P_S}V$ carries a residual $G$-action. The projection map $$\pi:G\times_{P_S}V\rightarrow G/P_S$$ $$[(g,v)]\mapsto [g]$$ then realizes $G\times_{P_S}V$ as a $G$-equivariant holomorphic vector bundle over $G/P_S$, called the associated bundle for $V$. In short, $G\times_{P_S}V$ is the $G$-equivariant holomorphic vector bundle over $G/P_S$ with fibre over $[e]\in G/P_S$ isomorphic to $V$ as a representation of $P_S$.

We now recall three standard group isomorphisms of importance to our work. Firstly, since $H^1(G/P_S,\mathcal{O}_{G/P_S})=0=H^2(G/P_S,\mathcal{O}_{G/P_S})$ (see \cite{Snow1989}), the exponential sequence gives rise to a group isomorphism $$\text{Pic}(G/P_S)\rightarrow H^2(G/P_S;\mathbb{Z})$$ $$[\mathbf F]\mapsto c_1(\mathbf F).$$ Secondly, recall that $$X^*(T)^{W_S}\rightarrow H^2(G/P_S;\mathbb{Z})$$ $$\beta\mapsto c_1(\mathcal{L}(\beta))$$ is a group isomorphism, where $\mathcal{L}(\beta)$ is the associated line bundle on $G/P_S$ arising from the $1$-dimensional $P_{S}$-representation of weight $\beta$. We therefore have a third isomorphism, $$X^*(T)^{W_S}\rightarrow\text{Pic}(G/P_S)$$ $$\beta\mapsto[\mathcal{L}(\beta)].$$ It will be prudent to recall a natural $\mathbb{Z}$-basis of the group $X^*(T)^{W_S}$. For $\beta\in\Pi$, let $h_{\beta}\in[\mathfrak{g}_{\beta},\mathfrak{g}_{-\beta}]$ be the corresponding simple coroot. Note that the $h_{\beta}$ form a basis of $\mathfrak{t}$ dual to the basis of fundamental weights $\omega_{\beta}\in X^*(T)$, $\beta\in\Pi$.

\begin{lemma}\label{Basis of Invariant Weight Lattice}
The group $X^*(T)^{W_S}$ has a $\mathbb{Z}$-basis of $\{\omega_{\beta}:\beta\in\Pi\setminus S\}$.
\end{lemma}

\begin{proof}
Note that $\delta\in X^*(T)$ belongs to $X^*(T)^{W_S}$ if and only if $\delta$ is fixed by each simple reflection $s_{\beta}$, $\beta\in S$. This holds if and only if $\delta$ is orthogonal to each simple root in $S$. The desired conclusion then follows from the fact that $$\delta=\sum_{\beta\in \Pi}\delta(h_{\beta})\omega_{\beta}=\sum_{\beta\in\Pi}2\frac{(\delta,\beta)}{(\beta,\beta)}\omega_{\beta}$$ is the expression of $\delta$ as a linear combination of the fundamental weights.
\end{proof}

We conclude this section with a proposition that will be of use later.

\begin{proposition}\label{Uniqueness of Corank-One Subbundles}
If $\mathbf F_1$ and $\mathbf F_2$ are $G$-invariant corank-$1$ subbundles of $\mathbf T_{G/P_{S}}$ and the quotients $\mathbf T_{G/P_{S}}/\mathbf F_1$ and $\mathbf T_{G/P_{S}}/\mathbf F_2$ are isomorphic as holomorphic line bundles, then $\mathbf F_1=\mathbf F_2$.
\end{proposition}

\begin{proof}
To begin, note that $\left(\mathbf T_{G/P_S}\right)_{[e]}$ is canonically isomorphic to $\mathfrak{g}/\mathfrak{p}_S$ as a $P_S$-representation, so that \begin{equation}\label{Tangent Bundle} \mathbf T_{G/P_S}\cong G\times_{P_S}(\mathfrak{g}/\mathfrak{p}_S).\end{equation} The isomorphism \eqref{Tangent Bundle} restricts to isomorphisms $$\mathbf F_1\cong G\times_{P_S}V_1$$ and $$\mathbf F_2\cong G\times_{P_S}V_2,$$ where $V_1$ and $V_2$ are codimension-$1$ $P_S$-subrepresentations of $\mathfrak{g}/\mathfrak{p}_S$. Since each $T$-weight space of \begin{equation}\label{Representation Isomorphism}
\mathfrak{g}/\mathfrak{p}_S=\bigoplus_{\beta\in\Delta^+\setminus\Delta_S^+}\mathfrak{g}_{\beta}
\end{equation}
is $1$-dimensional, each of $V_1$ and $V_2$ is obtained by removing a single weight space from the sum \eqref{Representation Isomorphism}. Let $\gamma_1,\gamma_2\in\Delta^+$ be the weights discarded to obtain $V_1$ and $V_2$, respectively. We then have bundle isomorphisms
$$\mathbf T_{G/P_{S}}/\mathbf F_1\cong G\times_{P_S}((\mathfrak{g}/\mathfrak{p}_S)/V_1)\cong\mathcal{L}(\gamma_1)$$ and $$\mathbf T_{G/P_{S}}/\mathbf F_2\cong G\times_{P_S}((\mathfrak{g}/\mathfrak{p}_S)/V_2)\cong\mathcal{L}(\gamma_2).$$ In particular, $\mathcal{L}(\gamma_1)\cong\mathcal{L}(\gamma_2)$ as holomorphic line bundles over $G/P_S$, so that $\gamma_1=\gamma_2$. Hence, $V_1=V_2$, implying that $\mathbf F_1$ and $\mathbf F_2$ identify with the same subbundle of $G\times_{P_S}(\mathfrak{g}/\mathfrak{p}_S)$ under \eqref{Tangent Bundle}. This completes the proof.
\end{proof}

In the case that $\mathbf F_1$ and $\mathbf F_2$ define contact structures, we have the following immediate

\begin{corollary}\label{Uniqueness of Contact Structure}
If each of $\mathbf F_1$ and $\mathbf F_2$ is the distribution of a $G$-invariant contact structure on $G/P_S$ and $\mbox{Pic}(G/P_S)\cong\mathbb Z$, then $\mathbf F_1=\mathbf F_2$.
\end{corollary}

This is simply the result of combining Proposition \ref{Uniqueness of Corank-One Subbundles} with the fact $G/P_S$ is Fano (and then applying the second observation listed at the end of Section \ref{Section-Review-Fano-Contact}.

\subsection{The Projectivization of the Minimal Nilpotent Orbit}\label{The Projectivization of the Minimal Nilpotent Orbit}
The material in \ref{Basic Setup} facilitates a worthwhile discussion of $\mathbb{P}(\mathcal{O}_{\text{min}})$ and its $G$-invariant contact structure. To this end, recall that the nilpotent cone of $\mathfrak{g}$ is the closed subvariety $$\mathcal{N}=\{\xi\in\mathfrak{g}:\adj(\xi)\text{ is nilpotent}\},$$ where $\adj:\mathfrak{g}\rightarrow\mathfrak{gl}(\mathfrak{g})$ is the adjoint representation of $\mathfrak{g}$. The nilpotent cone is $G$-invariant and consists of finitely many $G$-orbits, called nilpotent orbits. The set of nilpotent orbits is partially ordered according to the closure order, namely $\mathcal{O}_1\leq\mathcal{O}_2$ if and only if $\mathcal{O}_1\subseteq\overline{\mathcal{O}_2}$. The non-zero nilpotent orbits have a unique minimal element, $\mathcal{O}_{\text{min}}$, called the minimal nilpotent orbit. It is known that $\mathcal{O}_{\text{min}}$ is the $G$-orbit of a non-zero vector in the lowest root space $\mathfrak{g}_{-\lambda}$. 

Nilpotent orbits are invariant under scaling action of $\mathbb{C}^*$ on $\mathfrak{g}$. In particular, we have an inclusion of $\mathbb{P}(\mathcal{O}_{\text{min}}):=\mathcal{O}_{\text{min}}/\mathbb{C}^*$ into $\mathbb{P}(\mathfrak{g})$ as a closed $G$-orbit. Now, suppose that $\xi\in\mathfrak{g}_{-\lambda}\setminus\{0\}$, which determines a class $[\xi]\in\mathbb{P}(\mathcal{O}_{\text{min}})$. The $G$-stabilizer of $[\xi]$ is the standard parabolic subgroup $P_{\Lambda}$, where $\Lambda$ is the collection of those simple roots which are orthogonal to $\lambda$. We therefore have the $G$-variety isomorphism \begin{equation}\label{Isomorphism}\varphi:G/P_{\Lambda}\xrightarrow{\cong}\mathbb{P}(\mathcal{O}_{\text{min}})\end{equation} $$[g]\mapsto[\Adj(g)(\xi)].$$ 

It turns out that $\mathbb{P}(\mathcal{O}_{\text{min}})$ carries a distinguished $G$-invariant contact structure, $\mathbf E_{\text{min}}\subseteq T_{\mathbb{P}(\mathcal{O}_{\text{min}})}$. To obtain it, note that the Killing form on $\mathfrak{g}$ restricts to a $G$-equivariant variety isomorphism between $\mathcal{O}_{\text{min}}$ and a coadjoint orbit in $\mathfrak{g}^*$. The latter has the Kirillov-Kostant-Souriau symplectic structure, so that $\mathcal{O}_{\text{min}}$ is symplectic. The symplectic form on $\mathcal{O}_{\text{min}}$ has weight $1$ with respect to the $\mathbb{C}^*$-action, and Lemma 1.4 of \cite{Beauville1998} then gives the desired contact structure on $\mathbb{P}(\mathcal{O}_{\text{min}})$.

Using Remark 2.3 from \cite{Beauville1998}, one can more explicitly describe the bundle $\mathbf E_{\text{min}}$. Let $[\xi]\in\mathbb{P}(\mathcal{O}_{\text{min}})$ be the class of a lowest root vector, as above. Via the isomorphism \eqref{Isomorphism}, the fibre $(\mathbf E_{\text{min}})_{[\xi]}$ identifies with a codimension-$1$ subspace of $\mathfrak{g}/\mathfrak{p}_{\Lambda}$, the tangent space of $G/P_{\Lambda}$ at the identity coset. Now, note that $\mathfrak{p}_{\Lambda}\subseteq\mathfrak({g}_{-\lambda})^{\perp}$, the orthogonal complement of $\mathfrak{g}_{-\lambda}$ with respect to the Killing form. Our fibre is then given by \begin{equation}
\label{Fibre}(\mathbf E_{\text{min}})_{[\xi]}=(\mathfrak{g}_{-\lambda})^{\perp}/\mathfrak{p}_{\Lambda}.\end{equation} Since $\mathbf E_{\text{min}}$ is a $G$-invariant subbundle of $T_{\mathbb{P}(\mathcal{O}_{\text{min}})}$, \eqref{Fibre} can be used to determine the fibre of $\mathbf E_{\text{min}}$ over any point. 

\subsection{Reduction to the Case of a Maximal Parabolic}\label{Reduction to the Case of a Maximal Parabolic}
Let us begin to directly address the classification of partial flag varieties admitting $G$-invariant contact structures. To this end, assume $S\subseteq\Pi$ is such that $G/P_S$ admits a $G$-invariant contact structure $\mathbf E\subseteq\mathbf T_{G/P_S}$. In light of earlier remarks, we shall also assume that $b_2(G/P_S)=1$. This second assumption imposes a significant constraint on the subsets $S$ under consideration. Indeed, one has a Schubert cell decomposition of $G/P_S$ into $B$-orbits, $$G/P_S=\coprod_{[w]\in W/W_S}BwP_S/P_S,$$ so that $H^2(G/P_S;\mathbb{Z})$ is free of rank equal to the number of (complex) codimension-$1$ Schubert cells. Since the codimension of $BwP_S/P_S$ in $G/P_S$ is the length of a minimal-length coset representative in $[w]\in W/W_S$, the codimension-$1$ Schubert cells are those of the form $Bs_{\beta}P_S/P_S$, $\beta\in\Pi\setminus S$. Hence, the condition $b_2(G/P_S)=1$ implies that $\Pi\setminus S$ has cardinality $1$, so that $S=\Pi\setminus\{\alpha\}$ for some unique $\alpha\in\Pi$. In other words, $P_S$ is a maximal parabolic subgroup of $G$.

\subsection{The Contact Line Bundle on $G/P_S$}\label{The Contact Line Bundle on G/P} We now give a more explicit description of the $G$-invariant contact structure $\mathbf E$. Using the bundle isomorphism \eqref{Tangent Bundle}, we will regard the fibre $\mathbf E_{[e]}$ as a codimension-$1$ $P_S$-subrepresentation of $\mathfrak{g}/\mathfrak{p}_S.$ Of course, since $\mathbf E$ is a  $G$-invariant subbundle of $\mathbf T_{G/P_S}$, we also have \begin{equation}\label{Contact Structure}
\mathbf E\cong G\times_{P_S}\mathbf E_{[e]}.
\end{equation}
Hence, the contact line bundle $\mathbf L=\mathbf T_{G/P_S}/\mathbf E$ is given by \begin{equation}\label{Contact Line Bundle}
\mathbf L\cong G\times_{P_S}\left((\mathfrak{g}/\mathfrak{p}_S)/\mathbf E_{[e]}\right).
\end{equation}

Using \eqref{Contact Line Bundle}, one can describe $\mathbf L$ in terms of the isomorphism $X^*(T)^{W_S}\cong\text{Pic}(G/P_S)$.

\begin{proposition}\label{Description of the Contact Line Bundle}
If $\mathfrak{g}$ is simply-laced, then the highest root $\lambda$ belongs to $X^*(T)^{W_S}$ and $\mathbf L$ is isomorphic to $\mathcal{L}(\lambda)$.
\end{proposition}

\begin{proof}
We begin with two observations. Firstly, we have $\mathbf L\cong\mathcal{L}(\gamma)$ for some $\gamma\in X^*(T)^{W_S}$. Secondly, since $\mathfrak{g}$ is simply-laced, $\lambda$ is the unique dominant root. Proving the proposition will therefore amount to showing that $\gamma\in\Delta$, and that $\gamma$ is dominant. For the former, note that each $T$-weight of $\mathfrak{g}/\mathfrak{p}_S$ is a root. It follows that the weight of the quotient representation $(\mathfrak{g}/\mathfrak{p}_S)/E_{[e]}$ is also a root. Also, the isomorphisms \eqref{Contact Line Bundle} and $L\cong\mathcal{L}(\gamma)$ together imply that $\gamma$ is the weight of $\mathfrak{g}/\mathfrak{p}_S$, so that $\gamma$ must be a root.

To prove that $\gamma$ is dominant, we note that \begin{equation}
\label{Power of Contact Line Bundle}\mathcal{L}((n+1)\gamma)\cong \mathbf L^{\otimes (n+1)}\cong \mathbf K_{G/P_S}^{\vee},\end{equation} where $2n+1$ is the (complex) dimension of $G/P_S$. Also, the isomorphism \eqref{Tangent Bundle} yields 
$$\mathbf K_{G/P_S}^{\vee}=\wedge^{2n+1}\mathbf T_{G/P_S}\cong G\times_{P_S}(\wedge^{2n+1}(\mathfrak{g}/\mathfrak{p}_S)).$$ By \eqref{Representation Isomorphism}, the weight of $\wedge^{2n+1}(\mathfrak{g}/\mathfrak{p}_S)$ is $$\mu_S:=\sum_{\beta\in\Delta^+\setminus\Delta_S^+}\beta,$$ and we therefore have \begin{equation}
\label{Canonical Bundle Isomorphism}
\mathbf K_{G/P_S}^{\vee}\cong\mathcal{L}(\mu_S).\end{equation} Combining \eqref{Power of Contact Line Bundle} and \eqref{Canonical Bundle Isomorphism}, we conclude that $(n+1)\gamma=\mu_S$. Since $\mu_S$ is dominant, this implies that $\gamma$ is dominant.
\end{proof}

Before proceeding to the next section, we note that our arguments allow us to quickly recover the following well-known fact.

\begin{corollary}\label{Uniqueness of Type A Contact Structure}
The subbundle $\mathbf{E}_{\text{min}}\subseteq\mathbf{T}_{\mathbb{P}(\mathcal{O}_{\text{min}})}$ is the unique $G$-invariant contact structure on $\mathbb{P}(\mathcal{O}_{\text{min}})$.
\end{corollary}

\begin{proof}
Suppose that $\mathbf{F}\subseteq\mathbf{T}_{\mathbb{P}(\mathcal{O}_{\text{min}})}$is a $G$-invariant contact structure. Let us first assume $G$ to be of type $ADE$ (so that $\mathfrak g$ is simply-laced). Note that both $\mathbf{E}_{\text{min}}$ and $\mathbf{F}$ pull-back to $G$-invariant contact structures on $G/P_{\Lambda}$ under the isomorphism \eqref{Isomorphism}. By Proposition \ref{Description of the Contact Line Bundle}, it follows that both contact line bundles $\mathbf{T}_{\mathbb{P}(\mathcal{O}_{\text{min}})}/\mathbf{E}_{\text{min}}$ and $\mathbf{T}_{\mathbb{P}(\mathcal{O}_{\text{min}})}/\mathbf{F}$  pull-back to $\mathcal{L}(\lambda)$ under \eqref{Isomorphism}. In particular, these bundles are isomorphic, and Proposition \ref{Uniqueness of Corank-One Subbundles} then implies $\mathbf{F}=\mathbf{E}_{\text{min}}$.

If $G$ is not of type $ADE$, then $\text{Pic}(\mathbb{P}(\mathcal{O}_{\text{min}}))\cong\mathbb{Z}$.  Now, it follows from Corollary \ref{Uniqueness of Contact Structure} that $\mathbf{F}=\mathbf{E}_{\text{min}}$.
\end{proof}

\section{A Classification of $G$-Invariant Contact Structures on $G/P$}\subsection{The Main Theorem}We now consolidate the results presented in Sections \ref{Reduction to the Case of a Maximal Parabolic} and \ref{The Contact Line Bundle on G/P}. In light of Proposition \ref{Description of the Contact Line Bundle}, we will assume $G$ to be of type $ADE$ for the duration of this article. We then have the following relationship between the simple root $\alpha$ from \ref{Reduction to the Case of a Maximal Parabolic} and the highest root $\lambda$.

\begin{proposition}\label{Description of Alpha and Lambda}
The root $\alpha$ is the unique simple root not orthogonal to $\lambda$.
\end{proposition}

\begin{proof}
By Lemma \ref{Basis of Invariant Weight Lattice}, $X^*(T)^{W_S}$ is freely generated by $\omega_{\alpha}$. Since $\lambda\in X^*(T)^{W_S}$ by Proposition \ref{Description of the Contact Line Bundle}, it follows that $\lambda=k\omega_{\alpha}$ for some non-zero $k\in\mathbb{Z}$. Also, we may write $$k\omega_{\alpha}=\lambda=\sum_{\beta\in \Pi}\lambda(h_{\beta})\omega_{\beta}=\sum_{\beta\in\Pi}2\frac{(\lambda,\beta)}{(\beta,\beta)}\omega_{\beta}.$$ Hence, for $\beta\in\Pi$, we have $(\lambda,\beta)=0$ if and only if $\beta\neq\alpha$.  
\end{proof}

Before continuing, we note the following implication of Proposition \ref{Description of Alpha and Lambda} for partial flag varieties in type $A$.

\begin{corollary}\label{No Invariant Contact Structure}
Suppose that $G=\SL_n(\mathbb{C})$ with $n\geq 3$. There does not exist a partial flag variety $X$ of $\SL_n(\mathbb{C})$ with $b_2(X)=1$ admitting an $\SL_n(\mathbb{C})$-invariant contact structure. Equivalently, none of the Grassmannians $\emph{Gr}(k,\mathbb{C}^n)$, $1\leq k\leq n-1$, supports an $\SL_n(\mathbb{C})$-invariant contact structure.
\end{corollary}

\begin{proof}
By Proposition \ref{Description of Alpha and Lambda}, the existence of such an $X$ would imply that there was a unique simple root not orthogonal to the highest root $\lambda$. However, for $G=\SL_n(\mathbb{C})$, $n\geq 3$, there are exactly two simple roots not orthogonal to $\lambda$. The formulation in terms of Grassmannians follows from their being the partial flag varieties of $\SL_n(\mathbb{C})$ having $b_2=1$.
\end{proof}

\begin{remark}
Corollary \ref{No Invariant Contact Structure} has an interesting consequence when $n$ is an even positive integer. Indeed, the odd-dimensional projective space $\mathbb{P}^{n-1}$ is then isomorphic to the projectivization of the minimal nilpotent orbit of $\Sp_{n}(\mathbb{C})$. In particular, $\mathbb{P}^{n-1}$ admits an $\Sp_{n}(\mathbb{C})$-invariant contact structure. Yet, Corollary \ref{No Invariant Contact Structure} implies that this contact structure is not $\SL_{n}(\mathbb{C})$-invariant for $n\geq 4$. 
\end{remark}

Let us return to the matter at hand. Proposition \ref{Description of Alpha and Lambda} establishes that $S=\Pi\setminus\{\alpha\}$ is the collection of those simple roots which are orthogonal to $\lambda$, namely $S=\Lambda$. Hence, $G/P_S=G/P_{\Lambda}$, which is $G$-equivariantly isomorphic to $\mathbb{P}(\mathcal{O}_{\text{min}})$ via \eqref{Isomorphism}. It therefore remains to prove that \eqref{Isomorphism} is additionally an isomorphism of contact varieties, recalling that $G/P_S=G/P_{\Lambda}$ has the $G$-invariant contact structure $\mathbf E\subseteq\mathbf T_{G/P_{\Lambda}}$ fixed in \ref{Reduction to the Case of a Maximal Parabolic}.

\begin{theorem}\label{Main Theorem 2}
The map $\varphi:G/P_{\Lambda}\rightarrow\mathbb{P}(\mathcal{O}_{\text{min}})$ in \eqref{Isomorphism} is an isomorphism of contact varieties. 
\end{theorem}

\begin{proof}
We are claiming that $\varphi^*(\mathbf{E}_{\text{min}})$ coincides with $\mathbf{E}$ when the former is regarded as a subbundle of $\mathbf{T}_{G/P_{\Lambda}}$. Observing that each of $\varphi^*(\mathbf{E}_{\text{min}})$ and $\mathbf E$ is a $G$-invariant corank-$1$ subbundle of $\mathbf{T}_{G/P_{\Lambda}}$, Proposition \ref{Uniqueness of Corank-One Subbundles} allows us to reduce this to showing $\mathbf{T}_{G/P_{\Lambda}}/\varphi^*(\mathbf{E}_{\text{min}})$ and $\mathbf{T}_{G/P_{\Lambda}}/\mathbf{E}$ to be isomorphic as holomorphic line bundles. The second line bundle is isomorphic to $\mathcal{L}(\lambda)$ by Proposition \ref{Description of the Contact Line Bundle}, so we are further reduced to proving that the fibre $(\mathbf{T}_{G/P_{\Lambda}}/\varphi^*(\mathbf{E}_{\text{min}}))_{[e]}$ has weight $\lambda$ as a $T$-representation.

Let $d_{[e]}\varphi:(T_{G/P_{\Lambda}})_{[e]}\rightarrow (T_{\mathbb{P}(\mathcal{O}_{\text{min}})})_{[\xi]}$ (where $[\xi]=\varphi([e])$) be the differential of $\varphi$ at $[e]$. Since $\varphi$ is $T$-equivariant, $d_{[e]}\varphi$ is an isomorphism of $T$-representations. Furthermore, $d_{[e]}\varphi(\varphi^*(\mathbf{E}_{\text{min}})_{[e]})=(\mathbf{E}_{\text{min}})_{[\xi]}$, so that $(\mathbf{T}_{G/P_{\Lambda}}/\varphi^*(\mathbf{E}_{\text{min}}))_{[e]}$ and $(T_{\mathbb{P}(\mathcal{O}_{\text{min}})})_{[\xi]}/(\mathbf{E}_{\text{min}})_{[\xi]}$ are isomorphic $T$-representations. We also have an isomorphism $(T_{\mathbb{P}(\mathcal{O}_{\text{min}})})_{[\xi]}\cong\mathfrak{g}/\mathfrak{p}_{\Lambda}$ from Section \ref{The Projectivization of the Minimal Nilpotent Orbit}, under which $(\mathbf{E}_{\text{min}})_{[\xi]}$ identifies with $(\mathfrak{g}_{-\lambda})^{\perp}/\mathfrak{p}_{\Lambda}$. Putting everything together, we have \begin{equation}
\label{Two Isomorphisms}
(\mathbf{T}_{G/P_{\Lambda}}/\varphi^*(\mathbf{E}_{\text{min}}))_{[e]}\cong (T_{\mathbb{P}(\mathcal{O}_{\text{min}})})_{[\xi]}/(\mathbf{E}_{\text{min}})_{[\xi]}\cong\mathfrak{g}/(\mathfrak{g}_{-\lambda})^{\perp}.\end{equation} Since we have $$(\mathfrak{g}_{-\lambda})^{\perp}=\mathfrak{b}\oplus\bigoplus_{\beta\in\Delta^+\setminus\{\lambda\}}\mathfrak{g}_{\beta},$$ \eqref{Two Isomorphisms} implies that $(\mathbf{T}_{G/P_{\Lambda}}/\varphi^*(\mathbf{E}_{\text{min}}))_{[e]}$ is indeed the $1$-dimensional $T$-representation of weight $\lambda$.  
\end{proof}

\subsection{Example: The Grassmannian of Isotropic $2$-Planes in $\mathbb{C}^{2n}$}

We now describe a class of explicit examples that satisfy the hypotheses of Theorem \ref{Main Theorem}. To this end, let us set $G=\SO_{2n}(\mathbb{C})$ with $n\geq 4$. Given $\theta\in\mathbb{R}/(2\pi\mathbb{Z})$, consider the $2\times 2$ matrix $$R(\theta):=\begin{bmatrix}
\cos(\theta) & -\sin(\theta)\\
\sin(\theta) & \cos(\theta)
\end{bmatrix}.$$ For $\theta_1,\theta_2,\ldots,\theta_n\in\mathbb{R}/(2\pi\mathbb{Z})$, we define $R(\theta_1,\theta_2,\ldots,\theta_n)$ to be the $2n\times 2n$ block-diagonal matrix $R(\theta_1)\oplus R(\theta_2)\oplus\cdots\oplus R(\theta_n)$. Note that the $R(\theta_1,\theta_2,\ldots,\theta_n)$ constitute a maximal torus of the compact real form $\SO(2n)\subseteq\SO_{2n}(\mathbb{C})$. Let $T\subseteq SO_{2n}(\mathbb{C})$ be the complexification of this maximal torus. We then choose our collection of simple roots to be $\Pi:=\{\alpha_1,\alpha_2,\ldots,\alpha_n\}$, where $\alpha_j:T\rightarrow\mathbb{C}^*$ is defined by the property $$\alpha_j(R(\theta_1,\theta_2,\ldots,\theta_n))=e^{i(\theta_j-\theta_{j+1})}$$ for $j\in\{1,\ldots,n-1\}$, while $\alpha_n:T\rightarrow\mathbb{C}^*$ satisfies $$\alpha_n(R(\theta_1,\theta_2,\ldots,\theta_n))=e^{i(\theta_{n-1}+\theta_n)}.$$ The highest root $\lambda$ is then given by $$\lambda(R(\theta_1,\theta_2,\ldots,\theta_n))=e^{i(\theta_1+\theta_2)}.$$ Furthermore, the subset of simple roots orthogonal to $\lambda$ is $\Lambda=\Pi\setminus\{\alpha_2\}$.

Now, let $B:\mathbb{C}^{2n}\otimes\mathbb{C}^{2n}\rightarrow\mathbb{C}$ be the complexification of the dot product on $\mathbb{R}^{2n}$. One then has the Grassmannian of isotropic $2$-planes in $\mathbb{C}^{2n}$, $Gr_B(2,\mathbb{C}^{2n})$. More explicitly, $$Gr_B(2,\mathbb{C}^{2n}):=\{V\in Gr(2,\mathbb{C}^{2n}):V\subseteq V^{\perp}\},$$ where $V^{\perp}$ denotes the complement of $V\in Gr(2,\mathbb{C}^{2n})$ with respect to $B$. This is a partial flag variety of $SO_{2n}(\mathbb{C})$, and one can verify that the $SO_{2n}(\mathbb{C})$-stabilizer of \begin{equation}
\label{Definition of W} W:=\text{span}\{e_1+ie_2,e_3+ie_4\}\in\text{Gr}_B(2,\mathbb{C}^{2n})\end{equation} is precisely $P_{\Lambda}\subseteq SO_{2n}(\mathbb{C})$. Hence, we have an $SO_{2n}(\mathbb{C})$-equivariant variety isomorphism $$\SO_{2n}(\mathbb{C})/P_{\Lambda}\cong\text{Gr}_B(2,\mathbb{C}^{2n}).$$ By \eqref{Isomorphism}, we have another $\SO_{2n}(\mathbb{C})$-equivariant isomorphism \begin{equation}\label{Description for SO_{2n}} \mathbb{P}(\mathcal{O}_{\text{min}})\cong\text{Gr}_B(2,\mathbb{C}^{2n}),\end{equation} where $\mathcal{O}_{\text{min}}$ is the minimal nilpotent orbit of $\SO_{2n}(\mathbb{C})$. 

It remains to give the contact structure on $\text{Gr}_B(2,\mathbb{C}^{2n})$ for which \eqref{Description for SO_{2n}} is an isomorphism of contact varieties. In other words, it remains to find the unique $\SO_{2n}(\mathbb{C})$-invariant contact structure on $\text{Gr}_B(2,\mathbb{C}^{2n})$. To this end, let $\mathbf F$ denote the tautological bundle on $\text{Gr}_B(2,\mathbb{C}^{2n})$, whose fibre over $V\in\text{Gr}_B(2,\mathbb{C}^{2n})$ is $V$ itself. Note that $\mathbf F$ is a subbundle of the trivial bundle $\text{Gr}_B(2,\mathbb{C}^{2n})\times\mathbb{C}^{2n}$, so that we may consider the subbundle $\mathbf{F}^{\perp}$ of complements with respect to $B$. By definition, $\mathbf F\subseteq \mathbf F^{\perp}$, and we may define $$\mathbf E:=\Hom(\mathbf F,\mathbf F^{\perp}/\mathbf F).$$ Note that $\mathbf{E}$ is canonically a subbundle of $\Hom(\mathbf F,\mathcal O^{\oplus2n}/\mathbf F)$, the pullback to $Gr_B(2,\mathbb{C}^{2n})$ of $\mathbf{T}_{Gr(2,\,\mathbb{C}^{2n})}$. In fact, we have the inclusion $$\mathbf{E}\subseteq\mathbf{T}_{Gr_B(2,\,\mathbb{C}^{2n})}$$ of subbundles of $\Hom(\mathbf F,\mathcal O^{\oplus2n}/\mathbf F)$, giving rise to a short exact sequence \begin{equation}\label{Short Exact Sequence} 0\rightarrow\mathbf E\rightarrow\mathbf T_{\text{Gr}_B(2,\,\mathbb{C}^{2n})}\rightarrow\wedge^2(\mathbf F^{\vee})\rightarrow 0\end{equation} (see \cite{Guillemin-Sternberg1990}, Chapter 14). Since $\wedge^2(\mathbf F^{\vee})=\det(\mathbf{F}^{\vee})$ is a line bundle, $\mathbf{E}$ is a corank-$1$ subbundle of $\mathbf{T}_{Gr_B(2,\,\mathbb{C}^{2n})}$. Indeed, we have the following proposition.   

\begin{proposition}
The subbundle $\mathbf{E}\subseteq\mathbf{T}_{\emph{Gr}_B(2,\,\mathbb{C}^{2n})}$ is the unique $\SO_{2n}(\mathbb{C})$-invariant contact structure on $\emph{Gr}_B(2,\mathbb{C}^{2n})$.
\end{proposition}

\begin{proof}
By Proposition \ref{Uniqueness of Corank-One Subbundles} and the discussion at the end of Section \ref{Section-Review-Fano-Contact},  the $\SO_{2n}(\mathbb{C})$-invariant contact structure on $\text{Gr}_B(2,\mathbb{C}^{2n})$ is the unique subbundle of $\mathbf{H}\subseteq\mathbf{T}_{\text{Gr}_B(2,\,\mathbb{C}^{2n})}$ such $\mathbf{H}$ is $\SO_{2n}(\mathbb{C})$-invariant and $\mathbf{T}_{\text{Gr}_B(2,\,\mathbb{C}^{2n})}/\mathbf{H}$ is the ample generator of $\text{Pic}(\text{Gr}_B(2,\mathbb{C}^{2n}))$. Accordingly, it will suffice to prove that $\mathbf{E}$ possesses these two properties. For the former, note that $\mathbf{F}^{\perp}/\mathbf{F}$ is an $\SO_{2n}(\mathbb{C})$-invariant subbundle of $\mathcal O^{\oplus2n}/\mathbf{F}$. Hence, $\mathbf{E}=\Hom(\mathbf{F},\mathbf{F}^{\perp}/\mathbf{F})$ is an $\SO_{2n}(\mathbb{C})$-invariant subbundle of $\Hom(\mathbf{F},\mathcal O^{\oplus2n}/\mathbf{F})$, and therefore also of $\mathbf{T}_{\text{Gr}_B(2,\,\mathbb{C}^{2n})}$. For our second property, \eqref{Short Exact Sequence} gives a bundle isomorphism $$\mathbf{T}_{\text{Gr}_B(2,\,\mathbb{C}^{2n})}/\mathbf{E}\cong\det(\mathbf{F}^{\vee}).$$ The bundle $\det(\mathbf{F}^{\vee})$ is indeed the ample generator of $\text{Pic}(\text{Gr}_B(2,\mathbb{C}^{2n}))$, so our proof is complete. \end{proof}

We wish to conclude with a comparison of our presentation of the $\SO_{2n}(\mathbb{C})$-invariant contact structure on the isotropic Grassmannian to the one presented in \cite{Hwang2001} (pp.353--354), whose distribution we will denote by $\mathbf P$.  There, $\mbox{Gr}_B(2,\mathbb C^{2n})$ is given an alternative presentation, as a parameter space for lines in a hyperquadric of dimension $2n-2$.  If $\ell$ is line in the hyperquadric representating a point in the parameter space, and if we choose an isomorphism $\ell\cong\mathbb P^1$, the fibre $\mathbf P_\ell$ is the space of global holomorphic sections of the $(2n-4)$-fold direct sum of the hyperplane bundle on the $\mathbb P^1$.  One must choose an isomorphism for each point in order to describe $\mathbf P$ and so this description --- while explicit --- is local.  Our presentation of the unique $\SO_{2n}(\mathbb{C})$-invariant contact structure, with distribution $\mathbf E$ given above, does not depend on a family of isomorphisms and uses the tautological bundle on the isotropic Grassmannian directly.

\bibliographystyle{acm} 
\bibliography{fano}
\end{document}